\documentclass[12pt]{amsart}

\def\hmc{{\mathsf M}_p(n)}
\def\pcc{{\mathsf C}_p(n)}
\def\qc{{\mathcal A}_p(S_n)}

\def\hmco{{\mathsf M}_p(kp+1)}
\def\hmcot{{\mathsf M}_3(3k+1)}
\def\apg{{\mathcal A}_p(G)}

\def\ipn{{\mathsf I}_p(n)}
\def\lpn{{\mathsf L}_p(n)}
\def\lk{{\mathsf l}{\mathsf k}}
\def\pp{{\mathsf P}}
\def\ca{{\mathcal A}}
\def\cs{{\mathcal S}}
\def\cm{{\mathcal M}}
\def\cn{{\mathcal N}}
\def\cc{{\mathbb C}}
\def\mm{{\mathbf m}}
\def\e{{\mathsf e}}
\def\h{{\mathsf h}}
\def\f{{\mathsf f}}
\def\g{{\mathsf g}}
\def\p{{\mathsf p}}
\def\q{{\mathsf q}}
\def\c{{\mathsf c}}
\def\d{{\mathsf d}}
\def\s{{\mathsf s}}
\def\ov{\overline}
\def\op{\ov{P}}
\def\hz{\hat{0}}
\def\ho{\hat{1}}
\def\de{\Delta}
\def\rk{{\mathsf r}{\mathsf a}{\mathsf n}{\mathsf k}}
\def\rh{\widetilde{H}}
\def\ch{{\rm ch}}
\newtheorem{theorem}{Theorem}[section]
\newtheorem{lemma}[theorem]{Lemma}
\newtheorem{proposition}[theorem]{Proposition}
\newtheorem{corollary}[theorem]{Corollary}

\newtheorem{remark}[theorem]{Remark}
\newtheorem{conjecture}[theorem]{Conjecture}

\begin{document}

\title[Top homology]{Top homology of hypergraph matching complexes, $p$-cycle
complexes and Quillen complexes of symmetric groups}

\author[John Shareshian]{John Shareshian$^{1}$}
\footnotetext[1]{Supported by National Science Foundation grants
DMS 0300483 and DMS 0604233.}
\address{Department of Mathematics, Washington University, St.
Louis, MO 63130} \email{shareshi@math.wustl.edu}

\author[Michelle  Wachs]{Michelle L. Wachs$^{2}$}
\footnotetext[2]{Supported by National Science Foundation grants
DMS 0302310 and DMS 0604562.}
\address{Department of Mathematics, University of Miami, Coral
Gables, FL 33124-4250}
\email{wachs@math.miami.edu}


\begin{abstract}
We investigate the representation of a symmetric group $S_n$ on the homology of its Quillen complex at a prime $p$.  For homology groups in small codimension, we derive an explicit formula for this representation in terms of the
representations of symmetric groups on homology groups of $p$-uniform hypergraph matching complexes.  We conjecture an
explicit formula for the representation of $S_n$ on the top homology group of the corresponding hypergraph matching
complex when $n \equiv 1 \bmod p$.  Our conjecture follows from work of Bouc when $p=2$, and we prove the conjecture when
$p=3$.
\end{abstract}

\maketitle

\section{Introduction}

Our purpose is to obtain information on the homology of Quillen
complexes of symmetric groups (at odd primes) by studying
hypergraph matching complexes.  This method was first investigated
(for the prime two) by S. Bouc in \cite{Bo} and developed further
by R. Ksontini in \cite{Ks1,Ks2,Ks3} and later by Shareshian in
\cite{Sh}. We concentrate here on the top (and near top) homology of the two
complexes just named, along with that of a complex called the
$p$-cycle complex, which is essential for the transfer of
information that we obtain.  We work throughout this paper with complex coefficients.

Let us now define the objects with which we are concerned (further
basic definitions appear in Section \ref{defsec}) and provide some
more history and motivation along with a description of our
(somewhat technical) results.  Let $G$ be a finite group and let
$p$ be a prime.  The {\it Quillen complex} $\de\ca_p(G)$ is the
order complex of the partially ordered set $\ca_p(G)$ of all
nontrivial elementary abelian $p$-subgroups of $G$.  Interest in
these complexes was sparked by the paper \cite{Qu} of D. Quillen.
(Earlier work on the closely related Brown complex $\de\cs_p(G)$,
which is the order complex of the poset of all nontrivial
$p$-subgroups of a not necessarily finite group $G$, was done by
K. S. Brown in \cite{Br2,Br1}.)  Among many other things, it is
shown in \cite{Qu} that if $G$ is a group of Lie type in
characteristic $p$ then $\de\ca_p(G)$ is homotopy equivalent to
the building for $G$. Thus in this case the (reduced) homology of
$\de\apg$ is concentrated in a single dimension.  Moreover, as
noted in \cite{We}, it can be shown that the representation of $G$
on this unique nontrivial homology group obtained from the natural
action of $G$ on $\ca_p(G)$ is the same as that of $G$ on the
unique nontrivial homology group of the building, namely, the
Steinberg representation.

Given the results just mentioned, it is natural to investigate the
homology of Quillen complexes of symmetric and alternating groups.
(Note that if $p$ is odd then $\ca_p(S_n)=\ca_p(A_n)$.)  The
homology of $\de\ca_p(S_n)$ need not be concentrated in a single
dimension (see \cite[Section 15]{Ks1}).  It seems quite difficult
to determine the representation of $S_n$ on each nontrivial
homology group of its Quillen complex, or even the Betti numbers
of the complex.  Two alternatives suggest themselves in the search
for results analogous to those for the Lie type groups.  Namely,
one could investigate the Lefschetz virtual character (that is, the
alternating sum of the characters for the representations on
homology groups), as has been done successfully in various
combinatorial settings (see \cite{Bo2} for results along this line
for Quillen complexes of symmetric groups), or one could
investigate the top homology group. As mentioned above, we make
the second choice.

For a prime $p$ and an integer $n$, the {\it $p$-cycle complex}
$\pcc$ is a simplicial complex with one vertex for each subgroup
of $S_n$ that is generated by a $p$-cycle.  A collection of such
vertices forms a face of $\pcc$ if and only if the subgroups in
question together generate an abelian group. 
For any integers
$p,n$ the {\it hypergraph matching complex $\hmc$} is the
simplicial complex with one vertex for each subset of size $p$
from the $n$-set $[n]$, with a collection of such vertices forming a
face of $\hmc$ if and only if the sets in question are pairwise
disjoint.   One can view the vertices of $\hmc$ as  hyperedges of the complete $p$-uniform hypergraph on $[n]$ and the faces of $\hmc$ as  $p$-matchings on $[n]$.  For $p=2,3$ the complexes $\pcc$ and $\hmc$ are
isomorphic, as for each $X \subseteq [n]$ of size $p$, there is
unique cyclic group of order $p$ in $S_n$ having support $X$.
Matching complexes and related complexes
have been  studied in the literature  for their intrinsic 
combinatorial interest and in connection with applications
in various fields of mathematics,  see \cite {Wa} for a survey and
see \cite{Jo1,Jo2,Jo3, SW} for more recent developments.

As mentioned above, the idea of using $\hmc$ to study
$\de\ca_p(S_n)$ is originally due to Bouc, who studied the case
$p=2$.  
Various interesting results on
${\mathsf M}_2(n)$ appear in \cite{Bo}, including a complete
description of the representation of $S_n$ on its homology
groups (see (\ref{bouc})).  However, it seems quite difficult to use information
about ${\mathsf M}_2(n)$ to obtain results on $\de\ca_2(S_n)$.
Such a transfer of information is easier when the prime in
question is odd.  The first evidence of this appears in the thesis
\cite{Ks1} of Ksontini, where 
a relationship between  $\pcc$ and $\de\qc$ is discussed.  
There is an obvious simplicial map
from $\pcc$ to $\hmc$, induced by the map on vertices which sends
a cyclic group generated by a $p$-cycle to its support, and it is
natural to try to use this map to find useful relationships
between the topology of $\hmc$ and that of $\pcc$.  This is also done
quite successfully in \cite{Ks1}.  In \cite{Sh}, a result of A.
Bj\"orner, Wachs and V. Welker  \cite{BjWaWe} is used to make
a precise statement showing how the homotopy type of $\pcc$ is
completely determined by that of the complexes ${\mathsf M}_p(m)$
for all $m \leq n$ satisfying $m \equiv n \bmod p$. Since we are
interested here in representations on homology, we need an
$S_n$-equivariant homology version of this result. Such a result was
already proved in \cite{BjWaWe}.

It is straightforward to show that, for given $n,p$, the complexes
$\de\qc$, $\pcc$ and $\hmc$ all have dimension $t:=t(n,p):=\lfloor
\frac{n}{p} \rfloor-1$.  The first key idea for our work is an equivariant version of a result of Ksontini \cite[Proposition 8.1]{Ks1} stating that
if  $i$ is within $p-3$ of $t$ then
\begin{equation*} 
\rh_i(\de\qc)  \cong_{S_n} \rh_i(\pcc),
\end{equation*}
where $\cong_{S_n}$ denotes isomorphism of $\cc[S_n]$-modules. 
Since this result appears only in  the nonequivariant
form in the thesis of Ksontini  \cite{Ks1}, we provide a proof  in Section
\ref{qcpc} (see Theorem \ref{topcor}). 
In Section \ref{hmpc} we apply the result from
\cite{BjWaWe} in order to obtain a complicated but explicit
formula expressing the Frobenius characteristic of the
$\cc[S_n]$-module $\rh_*(\pcc)$ in terms of the characteristics of
the $\cc[S_m]$-modules $\rh_{*}({\mathsf M}_p(m))$ for all $m
\leq n$ such that $m \equiv n \bmod p$ (see Theorem \ref{fr}).

This leaves us with the problem of determining the representation
of $S_n$ on $\rh_{t-i}(\hmc)$ when $i \le p-3$.  In Section \ref{hyp} we give a
conjectural description of this representation when $i=0$ and $n=kp+1$ for
some $k$ (so $t=k-1$). Namely, Conjecture \ref{topcon} says that
$\rh_{k-1}({\mathsf M}_p(kp+1))$ is the $\cc[S_{kp+1}]$-submodule
generated by all simple submodules of the chain space
$C_{k-1}({\mathsf M}_p(kp+1))$ that are isomorphic to Specht
modules $\cs^\lambda$ where $\lambda$ has $k+1$ parts.  (As shown in
Section \ref{hyp}, if $C_{k-1}({\mathsf M}_p(kp+1))$ has a
submodule isomorphic to $\cs^\mu$ then $\mu$ has at most $k+1$
parts.)  It follows from the work of Bouc mentioned above that
Conjecture \ref{topcon} is true when $p=2$, and from work of Ksontini \cite{Ks3} that it is true when $k\le 2$.  We prove the
conjecture in the case $p=3$ by showing that the representation is isomorphic to the direct sum of
Specht modules indexed by partitions of $3k+1$ into $k+1$ odd parts.
As a corollary we have that  the representation of $S_{3k+1}$ on the top homology of the Quillen complex $\de \mathcal A_3(S_{3k+1})$ also has this nice decomposition.

\begin{center}

{\it Acknowledgments}

\end{center}

We thank the anonymous referee for helpful comments.

\section{Preliminaries: Definitions, notation and basic results} \label{defsec}

For a finite poset $P$, $\de P$ will denote the abstract
simplicial complex whose $k$-simplices are chains of length $k$
from $P$. For an abstract simplicial complex $\de$, $|\de|$ will
denote a geometric realization of $\de$ (note that all such
realizations are homeomorphic).  We will not
distinguish between $\de$ and $|\de|$.  Also, $\pp\de$ will denote
the poset of nonempty faces of a complex $\de$ and $\rh_i(\de)$
will denote the $i^{th}$ reduced simplicial homology group of
$\de$, with complex coefficients. A simplicial action of a group
$G$ on a complex $\de$ determines representations of $G$ on the
chain spaces and homology groups of $\de$.  (Note that an order
preserving action of $G$ on a poset $P$ determines a simplicial
action of $G$ on $\de P$.) We note here that the natural action of $S_n$ on $[n]$ induces a simplicial action
on $\hmc$ and that the action of $S_n$ by conjugation on the set of its $p$-subgroups induces simplicial actions on both $\pcc$ and $\de\qc$.

It is well known that for any complex
$\de$, $\de\pp\de$ is the barycentric subdivision of $\de$, so if
$G$ is a group acting (simplicially) on $\de$ then there is a
$G$-equivariant homeomorphism from $\de$ to $\de\pp\de$. The
next result follows, where $\cong_G$ denotes isomorphism of
$\cc[G]$-modules.

\begin{lemma}
Let $\de$ be a finite simplicial and let $G$ be a group acting
simplicially on $\de$.  Then for any integer $i$ we have
\[
\rh_i(\de) \cong_G \rh_i(\de\pp\de).
\]
\label{subdiv}
\end{lemma}

For a finite lattice $L$ with minimum element $\hz$ and maximum
element $\ho$,  let $\bar{L}$ denote the poset obtained from $L$
by removing $\hz$ and $\ho$ and let $\tilde L$  denote the subposet
of $\bar{L}$ consisting of those elements which can be obtained by
taking the meet (in $L$) of some set of maximal elements of
$\bar{L}$.  The next result is a well known equivariant version of
Rota's Cross-cut Theorem (\cite{Ro}).

\begin{lemma} Let $G$ be a group acting on a finite lattice $L$.
Then for every integer $i$, we have
\[
\rh_i(\de\bar{L}) \cong_G \rh_i(\de \tilde L).
\]
\label{maxint}
\end{lemma}

\begin{proof}
Let $\iota:\tilde L \rightarrow \bar{L}$ be the identity embedding.
Since $\iota$ is $G$-equivariant, our claim will follow from the
equivariant version of the Quillen fiber theorem, see
\cite[Corollary 9.3]{BjWaWe}, once we show that $\de\iota_{\geq y}$ is acyclic for all $y \in \bar{L}$, where $\iota_{\geq y}$ is the induced subposet $\{x \in \tilde L:x \geq y\}$ of $\tilde L$.   For any such $y$,
$\iota_{\geq y}$ contains a unique minimum element, namely, the
meet of all maximal elements of $\bar{L}$ which lie above $y$. 
Hence $\de\iota_{\geq y}$ is a cone and is therefore acyclic. 
\end{proof}

Let us now review some standard material on symmetric functions.  See \cite{MacD,St} for the relevant results. We will use the standard notation $\e,\h,\p,\s$ for the elementary, complete homogeneous, power sum, and Schur symmetric functions, respectively.  We hope that the use of $\p$ for a symmetric function and $p$ for a prime will not confuse the reader.  The irreducible representation of $S_n$ corresponding to  the partition $\lambda$ of $n$ will be denoted by $\cs^\lambda$.  The Frobenius characteristic is the isomorphism from the ${\mathbb N}_0$-graded ring  whose $n^{th}$ component consists of all (virtual and actual) representations of  $S_n$ to the graded ring of symmetric functions over the integers that sends $\cs^\lambda$ to $\s_\lambda$ for each partition $\lambda$.  In particular, the Frobenius characteristics of the trivial and alternating representations of $S_n$ are $\h_n$ and $\e_n$, respectively.  The Frobenius characteristic of a representation on the space $V$ will be denoted by $\ch V$

The plethysm of symmetric functions $\f,\g$ will be denoted by $\f[\g]$ rather than $\f \circ \g$ (which is used in \cite{MacD}).  Plethysm is used here to manipulate representations of symmetric groups induced from stabilizers of set partitions.  Say $n=kl$ with $1<k,l<n$.  The stabilizer $H$ in $S_n$ of a partition of $[n]$ into $k$ blocks $X_1,\ldots,X_k$ of size $l$ is isomorphic to the wreath product $S_k[S_l]$.  That is, $H$ contains the normal subgroup $K=\prod_{i=1}^{k}S_{X_i} \cong S_l^k$ (which we will call the {\it kernel} of $H$), and $H/K \cong S_k$.  We get a complement $C$ to $K$ in $H$ as follows. Fix for each $(i,j) \in [k] \times [k]$ a bijection $\psi_{ij}:X_i \rightarrow X_j$ such that
\begin{itemize}
\item $\psi_{ij}\psi_{jm}=\psi_{im}$ for all $i,j,m$, and
\item $\psi_{ji}=\psi_{ij}^{-1}$ for all $i,j$.
\end{itemize}
Define $\alpha:S_k \rightarrow S_n$ by $x\alpha(\tau):=x\psi_{i,i\tau}$ for $\tau \in S_k$ and $x \in X_i$.  Then $\alpha$ is an injective homomorphism, and it is straightforward to check that $C=Image (\alpha)$ is a complement to $K$ in $H$.  All such $C$ are conjugate in $H$, and we call such a $C$ a {\it standard complement} to $K$.  If $\phi:S_k \rightarrow GL(V)$ and $\rho:S_l \rightarrow GL(W)$ are representations, we get (after identifying $C$ with $S_k$ and each $S_{X_i}$ with $S_l$) a representation $\phi[\rho]:H \rightarrow GL(V \otimes \bigotimes^{k} W)$ such that for $g \in S_{X_i}$ we have
\[
(v \otimes w_1 \otimes \ldots \otimes w_i \otimes \ldots \otimes w_k)\phi[\rho](g) =
v \otimes w_1 \otimes \ldots \otimes w_i\rho(g) \otimes \ldots \otimes w_k,
\]
and for $c \in C$ we have
\[
(v \otimes w_1 \otimes \ldots \otimes w_k)\phi[\rho](c) =
v\phi(c) \otimes w_{1c} \ldots \otimes w_{kc}.
\]
If $\phi,\psi$ have Frobenius characteristics $\f,\g$, respectively, then the Frobenius characteristic of the induced representation
$\phi[\rho]\uparrow_{H}^{S_n}$ is $\f[\g]$.

\section{Relating the top homology of the $p$-cycle complex to
that of the Quillen complex} \label{qcpc}

Much (but not all) of the material in this section can be found in
\cite{Ks1} and \cite{Sh}. We include details here for the sake of
self-containment.

Say $P \in \qc$ has orbits $\Omega_1,\ldots,\Omega_r$ on $[n]$.
For each $i \in [r]$, we have a homomorphism
\begin{equation} \label{omi}
\omega_i:P \rightarrow S_{\Omega_i}
\end{equation}
determined by the action of $P$ on $\Omega_i$.  Set
\[
\ov{P}:=\prod_{i=1}^{r}\omega_i(P) \leq S_n.
\]
Each $\omega_i(P)$ is a quotient of the elementary abelian
$p$-group $P$ and is therefore an elementary abelian $p$-group.
Since elements of $S_n$ whose supports are disjoint commute, we
see that $\op$ is elementary abelian, that is, $\ov{P} \in \qc$.
Certainly $P \leq \ov{P}$.  Each $\omega_i(P)$ is a transitive
abelian subgroup of $S_{\Omega_i}$, and is therefore regular. In
other words, $|\omega_i(P)|=|\Omega_i|$. It follows that
\begin{equation} \label{rkeq}
\rk(\op)=\sum_{i=1}^{r}\rk(\omega_i(P))
=\sum_{i=1}^{r}\log_p(|\Omega_i|).
\end{equation}

\begin{lemma}
Let  $P \in \qc$.  Then $\rk(P) \leq
\lfloor \frac{n}{p} \rfloor$.\label{ranklem}
\end{lemma}

\begin{proof}
Our claim holds by observation when $n \leq p$, and we proceed by
induction on $n$.  It suffices to prove the claim when $P$ is such that $P=\ov{P}$.  Given the
orbits $\Omega_i$ and maps $\omega_i$ as defined above, we may assume that
$|\Omega_1|>1$.  Set
\begin{equation*}
Q=\prod_{i=2}^{r}\omega_i(P) \leq S_{\bigcup_{i=2}^{r}\Omega_i}.
\end{equation*}
By equation (\ref{rkeq}) and our inductive hypothesis, we have
\begin{equation}\label{rankq}
\rk(P)=\log_p(|\Omega_1|)+\rk(Q) \leq \log_p(|\Omega_1|)+\lfloor
\frac{n-|\Omega_1|}{p} \rfloor.
\end{equation}
 The lemma now
follows from the fact that for any positive
integer $a$, we have
\[
\lfloor \frac{n-p^a}{p} \rfloor +a \leq \lfloor \frac{n}{p}
\rfloor.
\]
\end{proof}

\begin{theorem}[Ksontini {\cite[Proposition 8.1]{Ks1}}]
Let $p$ be a prime and $n$ a positive integer.  Set $t=\lfloor
\frac{n}{p} \rfloor-1$.  Then
\begin{enumerate}
\item $\dim(\de\qc)=\dim(\pcc)=t$, and
\item $\rh_{j}(\de\qc) \cong_{S_n} \rh_{j}(\pcc)$ for all $j\ge t-(p-3)$.
\end{enumerate}
\label{topcor}
\end{theorem}

\begin{proof}
Set
\[
\cm=\{P \in \qc:P\mbox{ is generated by } \lfloor \frac{n}{p} \rfloor \mbox{ pairwise disjoint $p$-cycles}\},
\]
and let
\[
\ipn=\{Q \in \qc:Q \leq P \mbox{ for some } P \in \cm \}.
\]
Then $\ipn$ is an $S_n$-invariant subposet of $\qc$, and by Lemma
\ref{ranklem} we have
\[
\dim\de\qc=\dim\de\ipn=t.
\]
The elements of any  maximal face of $\pcc$ together generate an element of $\cm$.  It follows that
\[
\dim\pcc=t
\]
and claim (1) follows.

Next we show that if  $j \ge t-p+3$ then 
\begin{equation} \label{haiiso} \rh_j(\de\qc) \cong_{S_n} \rh_j(\de\ipn).
\end{equation} 
We prove this by first establishing the implication 
\begin{equation}\label{imp} P \in \qc-\ipn \implies \rk(P) \le t-(p-3).\end{equation}  It suffices to prove the implication when $P$ is a maximal element of  $\qc$, in which case $P =\ov P$.        Let $\Omega_i$  and $Q$ be defined as in the proof of Lemma~\ref{ranklem}, with $|\Omega_1| \ge | \Omega_i|$ for all $i$.   It follows from Lemma~\ref{ranklem} that   the string of equalities and inequalities given in (\ref{rankq}) holds.   Let $a =\log(|\Omega_1|)$.  Then  $a$ is a positive integer and  it follows from (\ref{rankq}) that 
$\rk(P) \le t +1 +a-p^{a-1}$.  It is easy to see that $a-p^{a-1} \le 2-p$ for all $a \ge 2$. Hence
$\rk(P) \le t -(p-3)$, unless $a=1$, which we claim is impossible.  Indeed, if $a=1$ then all orbits have size $p$ or $1$, which means that $P$ is generated by pairwise disjoint $p$-cycles.  Since $P$ is maximal, the number of $p$-cycles is $\lfloor {n\over p}\rfloor$, which means that $P \in \mathcal M$, contradicting $P \notin \ipn$.

It follows from the implication (\ref{imp}) that every chain  of  $\mathcal A_p(S_n)$ with  at least $t-p+4$ elements is a chain of $\ipn$ since its top element  has rank at least $t-p+4$.   
 Hence (\ref{haiiso}) holds.

Now let
\[
\lpn=\{Q \in \ipn:Q=\bigcap_{P \in \cn}P \mbox{ for some } \cn
\subseteq \cm \}.
\]
For $Q \in \lpn$, let $\cm_Q$ be the set of elements of $\cm$
which contain $Q$, so
\[
Q=\bigcap_{P \in \cm_Q}P.
\]
For (nonidentity) $q \in Q$, write
\[
q=\prod_{i=1}^{m}q_i,
\]
where the $q_i$ are pairwise disjoint $p$-cycles.  
Let $P \in \cm_Q$ be generated by pairwise disjoint
$p$-cycles $p_1,\ldots,p_t$.  The elements of $P$ have the  form
\[
\prod_{j=1}^{t}p_j^{a_j}, (0 \leq a_j<p),
\]
since the $p_j$ commute.
Since $q \in P$, we see that each
$q_i$ is a power of some $p_j$ and it follows that $q_i \in P$ for
each $i$.  Since $P$ is an arbitrary element of $\cm_Q$, we see
that each $q_i$ lies in $Q$.  Therefore, $Q$ is generated by
$p$-cycles.  We may now apply Lemma \ref{maxint} twice to get
\[
\rh_i(\de\ipn) \cong_{S_n} \rh_i(\de\lpn) \cong_{S_n}
\rh_i(\de\pp\pcc),
\] for all $i$,
and claim (2) now follows from the isomorphism (\ref{haiiso}) and Lemma \ref{subdiv}.
\end{proof}

\section{Relating the top homology of the $p$-cycle complex to
that of the hypergraph matching complex} \label{hmpc}

Let $\de$ be a simplicial complex on vertex set
$\{x_1,\ldots,x_n\}$ and let $\mm=(m_1,\ldots,m_n)$ be an
$n$-tuple of positive integers.  As defined in \cite[Section
6]{BjWaWe}, the $\mm$-inflation $\de_\mm$ of $\de$ is the complex
on vertex set $\{(x_i,j_i):i \in [n],j \in [m_i]\}$, such that
$\{(x_{i_1},j_{i_1}),\ldots,(x_{i_k},j_{i_k})\}$ is a
$(k-1)$-simplex in $\de_\mm$ if and only if
\begin{itemize}
\item $\{x_{i_1},\ldots,x_{i_k}\}$ is a $(k-1)$-simplex in $\de$,
and
\item $j_{i_l} \in [m_{i_l}]$ for all $l \in [k]$.
\end{itemize}
Roughly, $\de_\mm$ is obtained from $\de$ by taking $m_i$ copies
of vertex $x_i$ and then allowing every possible ``version" of a
face of $\de$ to be a face of $\de_\mm$.

The relevance of inflations to the matter at hand is made clear by
the following easy lemma (which is discussed in \cite{Sh}).

\begin{lemma}
Let $p$ be any prime and let $n$ be any positive integer.  Let
$\mm$ be the ${{n} \choose {p}}$-tuple each of whose entries is
$(p-2)!$. Then
\[
\pcc \cong \hmc_\mm.
\]
\label{inflem}
\end{lemma}

\begin{proof}
The vertices of $\hmc$ are the $p$-sets from $[n]$.  For each such
$p$-set $X$, there are $(p-2)!$ cyclic subgroups of order $p$ in
$S_n$ with support $X$.  A $(k-1)$-simplex of $\hmc$ is a
collection of $k$ disjoint $p$-sets, while a $(k-1)$-simplex of
$\pcc$ is a collection of $k$ cyclic subgroups of $S_n$ with
disjoint supports, each subgroup generated by a $p$-cycle.
\end{proof}

For a complex $\de$ and inflation $\de_\mm$, the homotopy type of
$\de_\mm$ is determined by the homotopy types of the links of the
faces of $\de$ (see \cite[Theorem 6.2]{BjWaWe}).  Thus the
homology of $\de_\mm$ is determined by the homology of links in
$\de$.  There is an equivariant version of this homology result,
which we will state below after making the appropriate
definitions.

When a group $G$ acts (simplicially) on a complex $\de$, $G_F$ will
denote the stabilizer in $G$ of a face $F$ of $\de$ and $\de/G$
will denote an arbitrary set of representatives for the orbits of
$G$ on the set of faces of $\de$ (including the empty face).  If
$\de_\mm$ is an inflation of $\de$ ($\mm=(m_1,\ldots,m_n)$) and
$F=\{x_{i_1},\ldots,x_{i_k}\}$ is a face of $\de$ then $\mm(F)$
will denote the $k$-tuple $(m_{i_1},\ldots,m_{i_k})$.  The {\it
deflating map} $\delta:\de_\mm \rightarrow \de$ is the simplicial
map induced by the function on vertex sets which sends $(x_i,j_i)$
to $x_i$.  For a face $F$ of $\de$, $\dot{F}$ will denote the set
of all subsets of $F$ (a subcomplex of $\de$). If the actions of
$G$ on the complexes $\de_\mm$ and $\de$ are intertwined by
the deflating map then $\dot{F}_{\mm(F)}$ is a $G_F$-invariant
subcomplex of $\de_\mm$. In any case, $\lk_\de(F)$ is a
$G_F$-invariant subcomplex of $\de$. Thus if $\delta$ intertwines
the given actions then for any integers $i,j$, the tensor product
\[
\rh_i(\dot{F}_{\mm(F)}) \otimes \rh_j(\lk_\de(F))
\]
is a $\cc[G_F]$-module.

\begin{lemma}[{\cite[Corollary 9.5]{BjWaWe}}]
Let $\de$ be a simplicial complex on vertex set
$\{x_1,\ldots,x_n\}$ and let $\mm$ be an $n$-tuple of positive
integers.  Let $G$ be a group which acts simplicially on $\de$ and
$\de_\mm$ so that these actions of $G$ are intertwined by the
deflating map $\delta:\de_\mm \rightarrow \de$.  Then for each
integer $r$, we have
\[
\rh_r(\de_\mm) \cong_G \bigoplus_{F \in
\de/G}\left(\rh_{|F|-1}(\dot{F}_{\mm(F)}) \otimes
\rh_{r-|F|}(\lk_\de F)\right)\uparrow_{G_F}^{G}.
\]
\label{equinfl}
\end{lemma}

For a prime $p$ and an integer $n$, the deflating map $\delta:\pcc
\rightarrow \hmc$ intertwines the actions of $S_n$ on the two
complexes, so Lemma \ref{equinfl} applies.  The following facts
are straightforward to prove.
\begin{itemize}
\item For each integer $k$, the group $S_n$ acts transitively on
the set of $k$-simplices of $\hmc$.  Therefore, $\hmc/S_n$
consists of one matching with $k$ hyperedges for $0 \leq k \leq
\lfloor \frac{n}{p} \rfloor$.
\item Let $F \in \hmc/S_n$ with $|F|=k$, so $F$ contains $k$
hyperedges $E_1,\ldots,E_k$.  Set $\Omega^+=\bigcup_{i=1}^{k}E_i$
and $\Omega^-=[n] \setminus \Omega^+$.
\begin{itemize}
\item The stabilizer $G_F$ of $F$ in $S_n$ is $H \times S_{\Omega^-}$,
where $H \leq S_{\Omega^+}$ is isomorphic to the wreath product
$S_k[S_p]$.  The action of the kernel $K \cong (S_p)^k$ of $H$ is the
componentwise action, that is, the $i^{th}$ component of $K$
permutes the $p$ vertices of $E_i$, while a standard complement $C \cong S_k$ in
$H$ permutes the $k$ hyperedges $E_1,\ldots,E_k$, as described in Section \ref{defsec}.
\item The link $\lk_{\hmc}(F)$ is isomorphic to ${\mathsf
M}_p(n-kp)$ (it is the $p$-uniform hypergraph matching complex on
vertex set $\Omega^-$).
\item The group $H$ acts trivially on $\Omega^-$ and therefore
acts trivially on $\lk_{\hmc}(F)$.  The group $S_{\Omega^-}$ acts
trivially on $\Omega^+$ and therefore acts trivially on
$\dot{F}_{\mm(F)}$.  Thus for any integer $r$, the action of $G_F$
on the module $\rh_{|F|-1}(\dot{F}_{\mm(F)}) \otimes
\rh_{r-|F|}(\lk_{\hmc}(F))$ is the usual tensor product action -
that is, for $h \in H$, $\sigma \in S_{\Omega^-}$, $v \in
\rh_{|F|-1}(\dot{F}_{\mm(F)})$ and $w \in
\rh_{r-|F|}(\lk_{\hmc}(F))$, we have $(v \otimes w)(h,\sigma)=vh
\otimes w\sigma$.

\end{itemize}

\end{itemize}

For $F \in \hmc/S_n$ and $r$ any integer, let $V_r(F)$ be the
$\cc[G_F]$-module $\rh_{|F|-1}(\dot{F}_{\mm(F)}) \otimes
\rh_{r-|F|}(\lk_{\hmc(F)})$, with action as described above.  We
are interested in the induced module $V_r(F)
\uparrow_{G_F}^{S_n}$, which is a direct summand in the
$\cc[S_n]$-module $\rh_r(\pcc)$ by Lemma \ref{equinfl}.  Basic
facts from the theory of induced modules give
\begin{eqnarray*}
V_r(F) \uparrow_{G_F}^{S_n} & \cong_{S_n} &
(V_r(F)\uparrow_{G_F}^{S_{\Omega^+} \times
S_{\Omega^-}})\uparrow_{S_{\Omega^+} \times S_{\Omega^-}}^{S_n} \\
& \cong_{S_n} &
(\rh_{|F|-1}(\dot{F}_{\mm(F)})\uparrow_{H}^{S_{\Omega^+}} \otimes
\rh_{r-|F|}(\lk_{\hmc}(F)))\uparrow_{S_{\Omega^+} \times
S_{\Omega^-}}^{S_n}.
\end{eqnarray*}

As noted above, if $|F|=k$ then the $\cc[S_{\Omega^-}]$-module
$\rh_{r-|F|}(\lk_{\hmc}(F))$ is equivalent to the
$\cc[S_{n-kp}]$-module $\rh_{r-|F|}({\mathsf M}_p(n-kp))$.  In
addition, if $r=\dim(\pcc)$ then $r-|F|=\dim({\mathsf
M}_p(n-kp))$. This gives the connection between the top homology
modules of $p$-cycle complexes and those of $p$-uniform hypergraph
matching complexes, which will be made more precise once we
understand the $\cc[S_{\Omega^+}]$-module
$\rh_{|F|-1}(\dot{F}_{\mm(F)})\uparrow_{H}^{S_{\Omega^+}}$.

Let us first understand the complex $\dot{F}_{\mm(F)}$.  If the
matching $F$ contains hyperedges $E_1,\ldots,E_k$ then (for $0
\leq l \leq k-1$) an $l$-simplex in $\dot{F}_{\mm(F)}$ is obtained
by selecting $l+1$ of the $k$ hyperedges of $F$, and for each of
the selected hyperedges $E_i$, selecting one of the $(p-2)!$
subgroups of order $p$ from $S_{\Omega^+}$ whose support is $E_i$.
It follows that $$\dot{F}_{\mm(F)}= \de_1*\ldots*\de_k$$ where $*$ denotes
 the join of complexes and 
 each $\de_i$ is a set of $(p-2)!$
points (equivalently, a wedge of $(p-2)!-1$ spheres of dimension
$0$). These points are the nontrivial $p$-subgroups of $S_{E_i}$. 

The action of $H \cong S_k[S_p]$ on $\dot{F}_{\mm(F)}$ is quite
transparent when the complex is represented as the join of the
subcomplexes $\de_i$.  Namely, a standard complement $C \cong S_k$ acts by
permuting the $k$ subcomplexes $\de_i$, and for any $i \in [k]$,
the $i^{th}$ component of the kernel $K \cong (S_p)^k$ acts on the
$(p-2)!$ points of $\de_i$ as it acts by conjugation on its
$(p-2)!$ Sylow $p$-subgroups.

Let $P$ be the poset of $p$-subgroups of $S_p$ including the trivial subgroup;
so $P$
has one minimum element $\hat 0_P$ and $(p-2)!$ maximal elements. The symmetric group $S_p$
acts on elements of $P$ by conjugation.  
  Clearly  $P$ is isomorphic to  the poset  $P_i$ of faces of $\de_i$ for each $i$, and
  the action of $S_p$ on $P$ is equivalent to the action of   $S_{E_i}$  on $P_i$.
  The action of $S_p$ on $P$ induces an obvious action of the wreath product  $S_k[S_p]$  on the $k$-fold direct product $P^{\times k}$.
   It is
straightforward to show that if $\hz$ is the (unique) minimum
element of  $P^{\times k}$, then 
\[
\pp(\de_1 \ast \ldots \ast \de_k) \cong P^{\times k} \setminus \{ \hz \}.
\]
and that the action of $H$ on $\pp(\de_1 \ast \ldots \ast \de_k)$ is equivalent
to the  action of   $S_k[S_p]$ on $P^{\times k} \setminus \{ \hz \}$.
By \cite[Proposition 2.7]{Su} the representation  of $S_k[S_p]$ on $\rh_{k-1}(\de P^{\times k} \setminus \{ \hz \})$ is ${\rm sgn}_k[\nu]$, where ${\rm sgn}_k$ is the alternating (or sign) representation 
of $S_k$ and $\nu$ is the representation of $S_p$ on $\rh_0(\de(P\setminus\{\hat 0_P\}))$
induced by the action of $S_p$ on $P$ described above.  Hence the  representation of $H$ on 
$\rh_{k-1}(\dot{F}_{\mm(F)})$ is equivalent to the representation of $S_k[S_p]$ on ${\rm sgn}_k[\nu]$.
Thus by viewing ${\rm sgn}_k[\nu]$ as a $\cc[H]$-module via the isomorphism between $H$ and $S_k[S_p]$, we have
 \begin{equation}\label{wreath} \rh_{k-1}(\dot{F}_{\mm(F)})\uparrow_H^{S_{\Omega^+}} \cong_{S_{\Omega^+}} {\rm sgn}_k[\nu]\uparrow_H^{S_{\Omega^+}}.\end{equation}

It is well known and straightforward to show that if $\de$ is a
complex consisting of $m$ points, $K \leq S_m$ permutes these
points transitively, $W$ is the permutation module for $\cc[K]$
associated with this action, and $T$ is the (unique) trivial
$\cc[K]$-submodule of $W$ then
\begin{equation} \label{permh}
\rh_0(\de) \cong_K W/T.
\end{equation}

\begin{theorem}
Let $N_p$ be the normalizer of a Sylow $p$-subgroup of $S_p$.  Let
$\f_p$ be the Frobenius characteristic of the induced character
$1_{N_p}^{S_p}$.  Then the $\cc[S_{\Omega^+}]$-module
$\rh_{|F|-1}(\dot{F}_{\mm(F)})\uparrow_{H}^{S_{\Omega^+}}$ is
isomorphic to the $\cc[S_{kp}]$-module whose character has
Frobenius characteristic
\[
\e_k[\f_p-\h_p].
\]
\label{frch1}
\end{theorem}

\begin{proof} This follows from  the isomorphisms (\ref{wreath}) and (\ref{permh}), 
and facts about Frobenius characteristic reviewed in Section \ref{defsec}, after noting that the permutation character associated with the action of a
group $X$ on the conjugates of a subgroup $Y$ by conjugation is
$1\uparrow_{N_X(Y)}^{X}$, where $N_X(Y)$ is the normalizer of $Y$ in $X$.
\end{proof}

Now we examine the symmetric function $\f_p$ of Theorem
\ref{frch1}.  By \cite[(7.119)]{St}, we know that $\f_p$ is the
cycle index $Z_{N_p}$, that is,
\[
\f_p=\sum_{g \in N_p}\p_{\rho(g)},  
\]
where $\rho(g)$ is the partition of $p$ determined by the cycle
shape of $g$.  It is well known and straightforward to show that
$N_p=CP$, where $P$ is cyclic of order $p$ (generated by a
$p$-cycle) and $C$ is cyclic of order $p-1$ (generated by a
$(p-1)$-cycle).  Moreover, each nonidentity element of $N_p$ is
contained in either $P$ or exactly one of the $p$ conjugates of
$C$.  For each divisor $d$ of $p-1$, the group $C$ contains
exactly $\phi(d)$ elements with one fixed point and
$\frac{p-1}{d}$ $d$-cycles, where $\phi$ is Euler's totient
function.  It now follows that
\begin{equation} \label{fp}
\f_p=\p_1^p+(p-1)\p_p+p\p_1\sum_{d|p-1}\phi(d)\p_d^{(p-1)/d}.
\end{equation}

We can also express $\f_p$ in terms of Schur functions, that is,
we can find a formula for the multiplicity of each irreducible
character of $S_p$ in $1\uparrow_{N_p}^{S_p}$.  First, we
introduce some notation used in \cite[Exercise 7.88]{St}.  Namely,
for an $n$-cycle $w$ in $S_n$ and an integer $m$, we denote by
$\psi_{m,n}$ the character of $S_n$ obtained by inducing from
$\langle w \rangle$ the character which maps $w$ to $e^{2\pi
im/n}$.

\begin{lemma}
For any prime $p$, we have
\[
1\uparrow_{N_p}^{S_p}=\psi_{0,p-1}\uparrow_{S_{p-1}}^{S_p}-\psi_{1,p}.
\]
\label{chardiff}
\end{lemma}

\begin{proof}
Note first that for any group $X$, any subgroup $Y \leq X$, any
character $\chi$ of $Y$ and any $g \in X$ we have
\begin{equation} \label{induce}
\chi\uparrow_Y^X(g)=\frac{1}{|Y|}|C_X(g)|\sum_{h \in g^X \cap
Y}\chi(h),
\end{equation}
where $C_X(g)$ is the centralizer of $g$ and $g^X$ is the conjugacy class of $g$ in $X$.  Writing
$N_p=CP$ as above, we have (by definition)
\[
\psi_{0,p-1}\uparrow_{S_{p-1}}^{S_p}=1\uparrow_C^{S_p}
\]
and
\[
\psi_{1,p}=\theta\uparrow_P^{S_p},
\]
where $\theta$ is a one dimensional character of $P$ which maps a generator to $e^{2\pi i \over p}$.
As noted above, for $c \in C \setminus \{1\}$ we have
\[
|c^{S_p} \cap N_p|=p|c^{S_p} \cap C|,
\]
and it follows from (\ref{induce}) that
\[
1\uparrow_{N_p}^{S_p}(c)=1\uparrow_C^{S_p}(c).
\]
Since no conjugate of $c$ lies in $P$, we have
\[
\psi_{1,p}(c)=0
\]
and
\[
1\uparrow_{N_p}^{S_p}(c)=
\psi_{0,p-1}\uparrow_{S_{p-1}}^{S_p}(c)-\psi_{1,p}(c).
\]
For $g \in P \setminus \{1\}$, we have
\[
\sum_{x \in g^{S_p} \cap P}\theta(x)=\sum_{j=1}^{p-1}e^{2\pi
ij/p}=-1.
\]
Using (\ref{induce}) twice, we get
\[
-\psi_{1,p}(g)=\frac{1}{p}|C_{S_p}(g)|(1)
=\frac{1}{p(p-1)}|C_{S_p(g)}|(p-1)=1\uparrow_{N_p}^{S_p}(g).
\]
Since no conjugate of $g$ lies in $S_{p-1}$, we have
\[
\psi_{0,p-1}\uparrow_{S_{p-1}}^{S_p}(g)=0
\]
and
\[
1\uparrow_{N_p}^{S_p}(g)=
\psi_{0,p-1}\uparrow_{S_{p-1}}^{S_p}(g)-\psi_{1,p}(g).
\]
Finally, we have
\begin{eqnarray*}
1\uparrow_{N_p}^{S_p}(1) & = & [S_p:N_p]  = (p-2)!=p(p-2)!-(p-1)!
\\ & = & [S_p:C]-[S_p:P]=
\psi_{0,p-1}\uparrow_{S_{p-1}}^{S_p}(1)-\psi_{1,p}(1).
\end{eqnarray*}
Since every element of $N_p$ is conjugate to an element of $P$ or
an element of $C$, we are done.
\end{proof}

We can now apply a result discovered by W. Kra\'skiewicz and J.
Weyman and independently by R. Stanley (see \cite[Corollary
8.10]{KrWe} or \cite[Exercise 7.88]{St}) to find the decomposition
of $\f_p$ into Schur functions.  Namely, for a standard Young
tableau $T$ with $n$ boxes, let $D(T)$ be the set of all $i \in
[n]$ such that $i+1$ appears in a row of $T$ below the row which
contains $i$.  Define
\[
maj(T):=\sum_{i \in D(T)}i. 
\]
Now, for integers $m,n,k$ and a partition $\lambda$ of $n$, let
$M_{m,k,\lambda}$ be the number of standard Young tableaux $T$ of
shape $\lambda$ which satisfy $maj(T) \equiv m \bmod k$.

\begin{lemma}[{\cite[Corollary 8.10]{KrWe}},{\cite[Exercise
7.88]{St}}] For integers $m,n$, we have
\[
\psi_{m,n}=\sum_{\lambda \vdash n}M_{m,n,\lambda}\cs^\lambda.
\]
\label{kwst}
\end{lemma}

\begin{corollary}
For any prime $p$, we have
\[
\f_p=\sum_{\lambda \vdash
p}(M_{0,p-1,\lambda}-M_{1,p,\lambda})\s_\lambda.
\]
\label{schurd}
\end{corollary}

\begin{proof}
It follows directly from Lemma \ref{kwst} that
\[
\psi_{1,p}=\sum_{\lambda \vdash p}M_{1,p,\lambda}\cs^\lambda.
\]
It also follows that
\[
\psi_{0,p-1}=\sum_{\rho \vdash p-1}M_{0,p-1,\rho}\cs^\rho.
\]
Now let $\lambda$ be a partition of $p$ and let $T$ be a standard
Young tableau of shape $\lambda$.  Let $T^\prime$ be the standard
Young tableau obtained by removing the box containing $p$ from
$T$. Then
\[
maj(T^\prime)=\left\{
\begin{array}{ll}
maj(T) & \mbox{if } p-1 \not\in D(T), \\ maj(T)-(p-1) & \mbox{if }
p-1 \in D(T).
\end{array}
\right.
\]
In particular,
\[
maj(T^\prime) \equiv maj(T) \bmod (p-1).
\]
It now follows from the branching rule (see for example
\cite[Theorem 2.8.3]{Sa}) that
\[
\psi_{0,p-1}\uparrow_{S_{p-1}}^{S_p}=\sum_{\lambda \vdash
p}M_{0,p-1,\lambda}\cs^\lambda,
\]
and our corollary follows from Lemma \ref{chardiff}.
\end{proof}

Collecting our results from this section and the last, we get the
following theorem.

\begin{theorem}
Let $n$ be a  nonnegative integer and   $p$ be a prime.  Let
$\q_{n,p,i},\c_{n,p,i},\d_{n,p,i}$ be the Frobenius characteristics of
the $\cc[S_n]$-modules
$\rh_{\dim \de\qc-i}(\de\qc)$, $\rh_{\dim \pcc-i}(\pcc)$ and
$\rh_{\dim \hmc-i}(\hmc)$, respectively.  Then for all $i \le p-3$,
\begin{eqnarray*}
q_{n,p,i} &=& \c_{n,p,i} \\ & = & \sum_{k=0}^{\lfloor {n \over p}\rfloor}\d_{n-kp,p,i}
\e_k[-\h_p+\p_1^p+(p-1)\p_p+p\p_1\sum_{d|p-1}\phi(d)\p_d^{(p-1)/d}]
\\ & = &  \sum_{k=0}^{{\lfloor {n \over p}\rfloor }}\d_{n-kp,p,i}
\e_k[\sum_{\scriptsize\begin{array}{c} \lambda \vdash
p\\ \lambda \neq (p) \end{array}}(M_{0,p-1,\lambda}-M_{1,p,\lambda})\s_\lambda].
\end{eqnarray*}
  \label{fr}
\end{theorem}

\begin{remark} \label{remform} Note  that the restriction $i \le p-3$ is needed only for the first equation  of Theorem~\ref{fr}. 
 Note also that  the theorem is vacuous for $p=2$ and  that the theorem refers only to top homology in the case $p=3$.   In the cases $p=2,3$, it is easy to check that the sum inside the  second plethysm vanishes giving $ \c_{n,p,i} = \d_{n,p,i}$, which also follows from the fact that $\mathsf C_p(n)$ and $\mathsf M_p(n)$ are isomorphic complexes when $p=2,3$.  \end{remark}

Say $p$ is any integer greater than $1$ that divides $n$.  Then $\hmc$ has dimension $t:=
\frac{n}{p}  -1$, but every face of dimension $t-1$ is
contained in a unique face of dimension $t$.  It follows that we
can reduce $\hmc$ to a complex of dimension $t-1$ by a series of
elementary collapses (see for example \cite{Co}), giving,
\begin{equation}\label{mpn} \rh_{\dim(\hmc)}(\hmc) = 0,\end{equation} 
unless $n=0$ in which case
$ \ch \rh_{\dim(\hmc)}(\hmc) = 1.$ 
The following result is thus a consequence of Theorem~\ref{fr}.

\begin{corollary} If $p$ is an odd prime that divides $n$ then
\begin{eqnarray*}
& & \hspace{-.6in}\ch \rh_{\dim(\de\qc)}(\de\qc)\\ &=& \ch \rh_{\dim(\pcc)}(\pcc)\\ &=& e_{n \over p}[-\h_p+\p_1^p+(p-1)\p_p+p\p_1\sum_{d|p-1}\phi(d)\p_d^{(p-1)/d}]\\ &=& \e_{n \over p}[\sum_{\scriptsize\begin{array}{c} \lambda \vdash
p\\ \lambda \neq (p) \end{array}}(M_{0,p-1,\lambda}-M_{1,p,\lambda})\s_\lambda].\end{eqnarray*} 
\end{corollary}

\begin{remark}

As noted above, we have $C_3(3k) \cong M_3(3k)$ for all $k$.  It follows from the first equality in Theorem \ref{fr} and (\ref{mpn}) that
\[
\rh_{k-1}(\de \mathcal A_3(S_{3k}))=0.
\]

This was observed previously in \cite{AsSm}, with attribution to J. Thompson.
\end{remark}

\section{The top homology of some hypergraph
matching complexes} \label{hyp}

In this section, we present a conjecture on the $\cc[S_n]$-module
structure of the top homology group of $\hmc$ in the case $n
\equiv 1 \bmod p$ and prove our conjecture in the cases  $p \leq 3$ and   ${n-1 \over p} \le 2$. 
For a symmetric function ${\mathsf
f}=\sum_{\lambda}a_\lambda \s_\lambda$ and a positive integer $r$, define
\[
{\mathsf f}|_r:=\sum_{l(\lambda)=r}a_\lambda \s_\lambda,
\] where $l(\lambda)$ is the number of 
 parts of the partition $\lambda$. 

\begin{conjecture}
For  integers $k\ge 1$, $ p>1$ and $n=kp+1$, we have
\[
\ch\rh_{k-1}(\hmc) = (\e_k[\h_p]\h_1)|_{k+1}.
\]
\label{topcon}
\end{conjecture}

Note that by Proposition \ref{zero} below and Pieri's rule (see for example
\cite[Theorem 7.15.7]{St}), Conjecture~\ref{topcon} implies that if the coefficient of $\s_\lambda$ in the expansion of $\ch\rh_{k-1}(\hmc)$ is nonzero, then $l(\lambda)=k+1$ and $\lambda_{k+1}=1$

Our conjecture in the case  $k =1$ is easy to verify. The complex ${\mathsf M}_p(p+1)$ is a discrete point set, and thus its $0^{th}$ chain space is a permutation module on $p$-sets and
its $0^{th}$ homology group is obtained by taking the
quotient of this permutation module by the trivial module.  It follows that 
\begin{equation}\label{basecase} \rh_0({\mathsf M}_p(p+1) )\cong_{S_{p+1}} \cs^{(p,1)}.\end{equation}
  By Pieri's rule, we have
$\s_{(p,1)} = (\e_1[\h_p] \h_1)|_2$.

Now let us consider that  case  $k=2$.  It is known (see \cite[I.8]{MacD}) that
\begin{equation} \label{e2}  \e_2[\h_p] = \sum_{\scriptsize\begin{array}{c} (\lambda_1,\lambda_2)\vdash 2p \\ \lambda_1,\lambda_2 \text{ odd}\end{array} } \s_{(\lambda_1,\lambda_2)}.\end{equation}
Hence by Pieri's rule, the case $k=2$ of the conjecture is equivalent to the following result.
\begin{theorem}[Ksontini {\cite[Lemma 3.5]{Ks3}}] \label{kson1} For all integers $p \ge 2$,
$$\ch \rh_{1}(\mathsf M_p(2p+1)) =\sum_{\scriptsize\begin{array}{c} (\lambda_1,\lambda_2)\vdash 2p \\ \lambda_1,\lambda_2 \,\,\rm{ odd}\end{array} } \s_{(\lambda_1,\lambda_2,1)}.$$
\end{theorem}
In \cite{Ks3} this result is stated with  the unnecessary hypothesis that $p$ is an odd prime.  We will need this result later and will give a proof that is slightly simpler than that of   \cite{Ks3} (but quite similar).

Using  results of C. Carre, we obtain an alternative formulation of Conjecture~\ref{topcon}. 
Since $\h_p^k - \e_k[\h_p] = \h_1^k[\h_p] -\e_k[\h_p] = (\h_1^k-\e_k)[\h_p]$
and $\h_1^k-\e_k$ is the Frobenius characteristic of a representation, we have that
$ \h_p^k - \e_k[\h_p] $ is the Frobenius characteristic of a representation.  Hence
$ \h_p^k - \e_k[\h_p] $ is Schur positive,  as are  $ \h_p^k$ and $\e_k[\h_p] $.    It follows that the coefficient of each Schur function in the Schur function expansion of
$\e_k[\h_p] $ is at most the corresponding coefficient in the expansion of $\h_p^k$.   Using  Pieri's rule repeatedly, we get the following result
(which is noted in \cite{Ca}).

\begin{proposition} \label{zero} For nonnegative integers $k,p,j,r$, if $r >k+\min\{j,1\}$ then
\[( \e_k[\h_p]h_j) |_r = 0.\]
\end{proposition}

 For a partition $\lambda=(\lambda_1\ge \lambda_2\ge \dots )$ and a
positive integer $r$, define $\lambda^{(r)}$ to be the partition 
obtained from $\lambda$ by adding $1$ to $\lambda_i$ for $1 \leq i
\leq r$, that is
$$\lambda^{(r)} = (\lambda_1+1\ge \lambda_2+1\ge \dots \ge \lambda_r+1 \ge \lambda_{r+1} \ge \dots).$$ (Here we are viewing the  partition $\lambda$ as an infinite sequence with all but the first $l(\lambda)$ entries equal to $0$.)  Equivalently, the Young diagram of $\lambda^{(r)}$ is
obtained from that of $\lambda$ by adding a box to each of the
first $r$ rows (including empty rows).  Define the linear map $\gamma_r$ on the ring of
symmetric functions by
\[
\gamma_r(\sum_{\lambda}a_\lambda \s_\lambda)=\sum_\lambda
a_\lambda \s_{\lambda^{(r)}}.
\]

\begin{theorem}[Carre, {\cite[Theorem 24(2)]{Ca}}]
For positive integers $p,r$, we have
\[
\e_r[\h_p]|_r=\gamma_r(\h_r[\h_{p-1}]).
\]
\label{carre}
\end{theorem}

By Pieri' s rule, Proposition~\ref{zero} and Theorem~\ref{carre}, we see that the following conjecture is equivalent to Conjecture~\ref{topcon}.
 
\begin{conjecture} \label{conj2} For  integers $k\ge 1$, $ p>1$ and $n=kp+1$, we have
\begin{equation} \label{secondconj}
\ch\rh_{k-1}(\hmc) = \gamma_{k+1}(\h_k[\h_{p-1}]).\end{equation}
\end{conjecture}

In the case $p=2$, equation (\ref{secondconj}) is equivalent to
\[ \rh_{k-1}(\mathsf M_2(2k+1)) \cong_{S_n} \cs^{(k+1,1^k)}, \]
which is a special case of  the following formula of Bouc\cite{Bo2} (see also \cite{DW}), giving the
homology of the matching complex  in each dimension,
\begin{equation} \label{bouc}  \rh_{r-1}(\mathsf M_2(n)) \cong_{S_n} 
\bigoplus_{\scriptsize\begin{array}{c}
\lambda:
\,\lambda
\vdash n
\\
\lambda  =
\lambda^\prime \\ d(\lambda) = n -2r \end{array}
} \cs^{\lambda},\end{equation}
where $\lambda^\prime$
denotes the conjugate of $\lambda$ and
$d(\lambda)$ denotes the size of the Durfee square of $\lambda$.

Before establishing the conjecture in the cases $p=3$ and $k=2$, we obtain some  information on $\rh_{k-1}(\mathsf M_p(n))$ for  arbitrary $p$, $n$ and $k = \lfloor {n \over p} \rfloor$.  
We understand the $\cc[S_n]$-module
$C_{r-1}(\hmc)$ for arbitrary $p,n$ and $r \leq \lfloor \frac{n}{p} \rfloor$.   It is induced from a one-dimensional (over
$\cc$) module $X$ of the stabilizer $G$ of an $(r-1)$-dimensional
face $F$. As noted above, this stabilizer is a direct product $H \times T$. Here $H$ is isomorphic to the wreath product $S_r[S_p]$, whose
kernel $K \cong S_p^r$ acts trivially on $X$.  A standard
complement $C$ to $K$ in $H$ permutes the $r$ components $S_p$ of
$K$ by conjugation in the same manner that it permutes the
vertices (hyperedges) of $F$, and $C$ acts on $X$ according to the
sign character of this permutation action.  The group $T \cong S_{n-rp}$ acts trivially on $X$.
We now have the following result.

\begin{proposition}
For integers $p,n \ge 2$ and $r \leq \lfloor \frac{n}{p} \rfloor$, the Frobenius characteristic of $C_{r-1}(\hmc)$ is given by
\[\ch C_{r-1}(\hmc) =
\e_r[\h_p]\h_{n-rp}.
\]
\label{cprop}
\end{proposition}

Now let  $p,n \ge 2$  and  $k = \lfloor \frac{n}{p} \rfloor$.  For $\lambda \vdash n$, define  $C_\lambda$ 
 to be the (direct) sum of all simple
submodules of the top chain space $C_{k-1}(\hmc)$ that are isomorphic to the Specht module $\cs^\lambda$
(and so have Frobenius characteristic $\s_\lambda$).  We can view $\rh_{k-1}(\hmc)$ as
a submodule of $C_{k-1}(\hmc)$ since $\dim\hmc=k-1$.

\begin{corollary}
The submodule $\bigoplus_{ l(\lambda)=k+1}C_\lambda$ of
$C_{k-1}(\hmc)$ is contained in $\rh_{k-1}(\hmc)$.
\label{submod}
\end{corollary}

\begin{proof}
If $l(\lambda)=k+1$ then $C_{k-2}({\mathsf M}_p(n))$ contains
no submodule isomorphic to $\cs^\lambda$ by Propositions~\ref{zero} and \ref{cprop} and Pieri's rule; so $C_\lambda$ is
contained in the kernel of the $(k-1)^{st}$ boundary map.
\end{proof}

\begin{corollary} \label{submodcor} For  integers $k\ge 1$, $ p>1$ and $n=kp+1$, we have
$$ 
\ch\rh_{k-1}(\hmc) = \gamma_{k+1}(\h_k[\h_{p-1}]) + \mathsf f = (\e_k[\h_p]\h_1)|_{k+1} + \mathsf f ,$$
where $\mathsf f$ is a symmetric function such that $\mathsf f|_r = 0$ for all $r \ge k+1$. Thus Conjecture~\ref {topcon}  (and \ref{conj2})  holds if and only if $\mathsf f =0$.  
\end{corollary}

The following long exact sequence appears in the
thesis \cite{Ks1} of Ksontini and is a generalization of a
sequence used by Bouc for graph matching complexes in \cite{Bo}.
It is straightforward to show that this sequence is the standard
long exact sequence associated with the ($S_{n-1}$-equivariant)
embedding of ${\mathsf M}_p(n-1)$ into $\hmc$ determined by the
identity embedding of $[n-1]$ into $[n]$.

\begin{lemma}[{\cite[Proposition 4.12]{Ks1}}]
For positive integers $p,n>1$ there is a long exact sequence
\[
\begin{array}{l}
\ldots \rightarrow \rh_r({\mathsf M}_p(n-1)) \rightarrow
\rh_r(\hmc)\downarrow_{S_{n-1}}^{S_n} \rightarrow
\rh_{r-1}({\mathsf M}_p(n-p)) \uparrow_{S_{n-p} \times
S_{p-1}}^{S_{n-1}} \\ \rightarrow \rh_{r-1}({\mathsf M}_p(n-1))
\rightarrow \rh_{r-1}(\hmc)\downarrow_{S_{n-1}}^{S_n} \rightarrow
\ldots
\end{array}
\]
of $\cc[S_{n-1}]$-modules, where the action of the component
$S_{p-1}$ on $\rh_\ast({\mathsf M}_p(n-p))$ is trivial.
\label{leslem}
\end{lemma}

By (\ref{mpn}) we have the following special case.
\begin{corollary}
For positive integers $k,p>1$, the long exact sequence of Lemma
\ref{leslem} with $n=kp+1$ is
\[
\begin{array}{l}
0 \rightarrow \rh_{k-1}(\hmco)\downarrow_{S_{kp}}^{S_{kp+1}}
\rightarrow \rh_{k-2}({\mathsf
M}_p((k-1)p+1))\uparrow_{S_{(k-1)p+1} \times S_{p-1}}^{S_{kp}} \\
\rightarrow \rh_{k-2}({\mathsf M}_p(kp)) \rightarrow \ldots.
\end{array}
\]
\label{les2}
\end{corollary}

\begin{proof}[Proof of Theorem~\ref{kson1}]  By Corollary \ref{les2}, we have a long exact sequence
\[
0 \rightarrow \rh_{1}({\mathsf
M}_p(2p+1))\downarrow_{S_{2p}}^{S_{2p+1}} \rightarrow
\rh_{0}({\mathsf M}_p(p+1))\uparrow_{S_{p+1} \times
S_{p-1}}^{S_{2p}} \rightarrow \ldots.
\] 
By (\ref{basecase}) and Pieri's rule, the $\cc[S_{2p}]$-module $\rh_{0}({\mathsf M}_p(p+1))\uparrow_{S_{p+1} \times
S_{p-1}}^{S_{2p}} $ is isomorphic to the direct sum of Specht modules $\cs^\mu$ over all partitions $\mu \vdash 2p$ of length 2 or 3, such that if $l(\mu) = 3$ then the smallest part of $\mu $ must be $1$.  Each Specht module has multiplicity 1 in this decomposition.  Our exact sequence shows that
$\rh_{1}({\mathsf
M}_p(2p+1))\downarrow_{S_{2p}}^{S_{2p+1}}$ embeds into
$\rh_{0}({\mathsf M}_p(p+1))\uparrow_{S_{p+1} \times
S_{p-1}}^{S_{2p}} $.  Hence each Specht module  in  $\rh_{1}({\mathsf
M}_p(2p+1))\downarrow_{S_{2p}}^{S_{2p+1}}$ must be of the form just described and must also have multiplicity $1$.  

Assume for contradiction that $\rh_{1}({\mathsf
M}_p(2p+1))$ has a submodule isomorphic to $\cs^\lambda$, where $l(\lambda) \le 2$.  Since the restriction of $\cs^\lambda$ must be isomorphic to a submodule of $\rh_{0}({\mathsf M}_p(p+1))\uparrow_{S_{p+1} \times
S_{p-1}}^{S_{2p}} $, the partition $\lambda$ can't have length 1; so it must have length 2 and consist of an even part and an odd part.   By reducing the even part by one, we get a partition $\mu\vdash 2p $ with 2 odd parts, $\mu_1,\mu_2$.   By the branching rule, the restriction $\rh_{1}({\mathsf
M}_p(2p+1))\downarrow_{S_{2p}}^{S_{2p+1}}$ has a submodule isomorphic to $\cs^\mu$.

Let   $\tau=(\mu_1,\mu_2,1)$.  By 
Corollary~\ref{submodcor}, equation~(\ref{e2}) and Pieri's rule, $\rh_{1}({\mathsf
M}_p(2p+1))$ has a submodule isomorphic to $\cs^\tau$.  Hence by the branching rule, $\rh_{1}({\mathsf
M}_p(2p+1))\downarrow_{S_{2p}}^{S_{2p+1}}$ has an additional submodule isomorphic to $\cs^\mu$, contradicting the multiplicity $1$ requirement.   

It follows that if $\cs^\lambda$ is isomorphic to a submodule of  $\rh_{1}({\mathsf
M}_p(2p+1))$, then $\lambda$ has length 3. 
Hence by Corollary~\ref{submodcor}
$$\ch \rh_{1}({\mathsf
M}_p(2p+1)) = \e_2[\h_p]h_1 |_{3} = \sum_{\scriptsize\begin{array}{c} (\lambda_1,\lambda_2)\vdash 2p \\ \lambda_1,\lambda_2 \text{ odd}\end{array} } \s_{(\lambda_1,\lambda_2,1)}$$
\end{proof}

Now we turn our attention to the case $p=3$ of the conjecture.  For a partition
$\lambda=(\lambda_1,\ldots,\lambda_t)$, write $2\lambda$ for the
partition  $(2\lambda_1,\ldots,2\lambda_t)$.  It is known (see for
example \cite[Example A2.9]{St}) that
\[
\h_k[\h_2]=\sum_{\lambda \vdash k}\s_{2 \lambda}.
\]
Hence 
\begin{equation}\label{oddeq} \gamma_{k+1}( \h_k[\h_2]) = \sum_{\lambda \in \Lambda(k)} \s_\lambda,
\end{equation}
where $\Lambda(k)$ is the set of all partitions of $3k+1$ into $k+1$ odd
parts. 
It follows that  the next result is the  case $p=3$ of  Conjecture~\ref{conj2}.
\begin{theorem}
Let $k$ be any positive integer.  Then 
\[ \rh_{k-1}(\mathsf M_3(3k+1)) \cong_{S_{3k+1}} 
\bigoplus_{\lambda \in \Lambda(k)} \cs^\lambda.
\]
 \label{oddprop}
\end{theorem}

\begin{proof}
We may assume $k\ge 3$  since the result for $k \le 2$ has already been established  by the isomorphism (\ref{basecase}) and  Theorem~\ref{kson1}.  By Corollary \ref{submodcor} and equation (\ref{oddeq}), we need only show that if $\rh_{k-1}(\hmcot)$ has a submodule
isomorphic to $\cs^\lambda$ then $\lambda$ has at least $k+1$ parts.   By Corollary~\ref{les2}, we have a long exact sequence
\[
0 \rightarrow \rh_{k-1}({\mathsf
M}_3(3k+1))\downarrow_{S_{3k}}^{S_{3k+1}} \rightarrow
\rh_{k-2}({\mathsf M}_3(3k-2))\uparrow_{S_{3k-2} \times
S_2}^{S_{3k}} \rightarrow \ldots.
\]
By inductive hypothesis, the $\cc[S_{3k-2}]$-module
$\rh_{k-2}({\mathsf M}_3(3k-2))$ is the direct sum of the Specht
modules $\cs^\lambda$ over all partitions $\lambda$ of $3k-2$ into
$k$ odd parts.  It follows from Pieri's rule that if $\mu$ is a
partition of $3k$ and $\rh_{k-2}({\mathsf
M}_3(3k-2))\uparrow_{S_{3k-2} \times S_2}^{S_{3k}}$ has a
submodule isomorphic to $\cs^\mu$ then
\begin{itemize}
\item $\mu$ has either $k$ or $k+1$ parts, and \item $\mu$ has at
most two even parts.
\end{itemize}

Our exact sequence shows that
$\rh_{k-1}(\hmcot)\downarrow_{S_{3k}}^{S_{3k+1}}$ embeds into
$\rh_{k-2}({\mathsf M}_3(3k-2))\uparrow_{S_{3k-2} \times
S_2}^{S_{3k}}$, and it follows from the branching rule (and simple
arithmetic) that if $\lambda$ is a partition of $3k+1$ and
$\rh_{k-1}(\hmcot)$ has a submodule isomorphic to $\cs^\lambda$ then
\begin{itemize}
\item $\lambda$ has at least $k$ parts, and \item if
$\lambda=(\lambda_1 \ge \dots \ge \lambda_k)$ then $\lambda_k \in
\{2,3\}$.
\end{itemize}
(Note that when $\lambda$ as above has $k$ parts, we cannot have
$\lambda_k=1$ since then an irreducible constituent of the
restriction of $\cs^\lambda$ has $k-1$ parts.)

If $\lambda=(\lambda_1,\ldots,\lambda_k)$ is a partition of $3k+1$
with $\lambda_k=3$ then $\lambda=(4,3,\ldots,3)$.  Assume for
contradiction that in this case $\rh_{k-1}(\hmcot)$ has a submodule
isomorphic to $\cs^\lambda$.  The restriction of this submodule to
$S_{3k}$ has a submodule isomorphic to $\cs^{(3,\ldots,3)}$ by the
branching rule.  By Corollary  \ref{submodcor} and equation (\ref{oddeq}), we know that
$\rh_{k-1}(\hmcot)$ has a submodule isomorphic to
$\cs^{(3,\ldots,3,1)}$ and the restriction of this submodule to
$S_{3k}$ produces an additional submodule isomorphic to
$\cs^{(3,\ldots,3)}$.  On the other hand, by Pieri's rule,
$\rh_{k-2}({\mathsf M}_3(3k-2))\uparrow_{S_{3k-2} \times
S_2}^{S_{3k}}$ has a unique submodule isomorphic to
$\cs^{(3,\ldots,3)}$, namely, a constituent of the module induced
from $\cs^{(3,\ldots,3,1)}$.  This gives the desired contradiction.

Finally, say $\lambda=(\lambda_1,\ldots,\lambda_k)$ is a partition
of $3k+1$ with $\lambda_k=2$. If $k=3$ then the only possibilities for  $\lambda$  are $ (6,2,2), (5,3,2)$ and $(4,4,2)$.  The Young diagram  $(5,2,2)$ can be obtained from either of the first two possibilities by removing a cell, and the  Young diagram  $(4,4,1)$ can be obtained from the third possibility by removing a cell.   However, again using Pieri's rule we see that no  submodule of $\rh_{k-2}({\mathsf
M}_3(3k-2))\uparrow_{S_{3k-2} \times S_2}^{S_{3k}}$ is isomorphic to $\cs^{(5,2,2)}$ or $\cs^{(4,4,1)}$.
 (One cannot add two boxes in the
same column of a Young diagram, so one cannot obtain two equal
even parts from a partition into odd parts.)

Now suppose $k>3$.  Since no restriction of $\cs^\lambda$ can have more than $2$ even parts, it follows
that $\lambda_k=2$ is the only even part of $\lambda$, and
$\lambda=(5,3,\ldots,3,2)$.  Now $\cs^{(5,3,\ldots,3,2,2)}$ is a
submodule of the restriction of $\cs^\lambda$.  However, again using
Pieri's rule, we see that no submodule of $\rh_{k-2}({\mathsf
M}_3(3k-2))\uparrow_{S_{3k-2} \times S_2}^{S_{3k}}$ is isomorphic
to $\cs^{(5,3,\ldots,3,2,2)}$.   By this final
contradiction, we see that if $\rh_{k-1}(\hmcot)$ has a submodule
isomorphic to $\cs^\lambda$ then $\lambda$ has at least $k+1$ parts.
\end{proof}

The next result now follows from Theorem~\ref{topcor} and the isomorphism between the complexes $\mathsf C_3(n)$ and $\mathsf M_3(n)$.

\begin{corollary} Let $k$ be any positive integer.  Then 
\[ \rh_{k-1}(\de  \mathcal A_3(3k+1)) \cong_{S_{3k+1}}  \rh_{k-1}(\de  \mathsf C_3(3k+1)) \cong_{S_{3k+1}} 
\bigoplus_{\lambda \in \Lambda(k)} \cs^\lambda.
\]
\end{corollary}

The following result of  Athanasiadis on  nonvanishing homology of $\mathsf M_n(p)$ enables one to 
obtain   precise information on   $\rh_r(\mathsf M_n(p))$ when $n$ is small relative to $p$.

\begin{theorem}[Athanasiadis \cite{At}] \label{athan} For $n,p \ge 2$, the
homology of $\hmc$ vanishes below dimension $ \lfloor {n  -p\over p+1} \rfloor$.
\end{theorem}

\begin{corollary} \label{athancor} If $k \le p+3$ then the homology of $\hmco$  is nonvanishing in at most two  dimensions.
\end{corollary}

We obtain the table below giving all nonvanishing $\rh_r(\mathsf M_3(n))$ for  $ n
\leq 13$ by  using Maple to compute the right hand side of  the equivariant Euler-Poincar\'e formula, 
$$ \sum_{r=0}^{\lfloor {n \over 3}\rfloor} (-1)^r \ch \rh_{r-1}({\mathsf M}_3(n)) = \sum_{r = 0}^{\lfloor {n \over 3}\rfloor} (-1)^r \e_r[\h_3] \h_{n-3r} .$$  The computation of homology for $n \neq 10,13$ then follows from Theorem~\ref{athan}, which guarantees that  the left hand side has at most one nonzero term. When $n=10,13$, the computation  follows from Theorem~\ref{oddprop} and Corollary~ \ref{athancor}.

\[
\begin{array}{|c|c||c|}
\hline n & r & \rh_r({\mathsf M}_3(n)) \\ \hline \hline 4 & 0 & \cs^{(3,1)}
\\ \hline 5 & 0 & \cs^{(4,1)} \oplus \cs^{(3,2)} \\ \hline 6 & 0 &
\cs^{(4,2)} \\ \hline 7 & 1 & \cs^{(5,1,1)} \oplus \cs^{(3,3,1)} \\
\hline 8 & 1 & \cs^{(6,1,1)} \oplus \cs^{(5,2,1)} \oplus \cs^{(4,3,1)}
\oplus \cs^{(3,3,2)} \oplus \cs^{(5,3)} \\ \hline 9 & 1 & \cs^{(6,2,1)}
\oplus \cs^{(5,3,1)} \oplus \cs^{(4,3,2)} \oplus \cs^{(5,4)} \\ \hline
10 & 1 & \cs^{(5,5)} \\ 
\hline 10 & 2 & \cs^{(7,1,1,1)} \oplus
\cs^{(5,3,1,1)} \oplus \cs^{(3,3,3,1)} \\
 \hline11& 2 & \cs^{(8,1,1,1)}\oplus \cs^{(7,3,1)}\oplus \cs^{(7,2,1,1)}\oplus \cs^{(6,4,1)}\oplus \cs^{(6,3,2)}
      \oplus \cs^{(6,3,1,1)}\\ & &  \oplus \cs^{(5,4,2)}\oplus \cs^{(5,4,1,1)}\oplus \cs^{(5,3,3)} \oplus \cs^{(5,3,2,1)}
  \oplus \cs^{(4,3,3,1)} \oplus \cs^{(3,3,3,2)}\\
 \hline 12 &2 & \cs^{(8,2,1,1)}\oplus \cs^{(7,4,1)} \oplus \cs^{(7,3,2)}\oplus \cs^{(7,3,1,1)} \oplus \cs^{(6,5,1)} \oplus \cs^{(6,4,2)} \\ & &  \oplus \cs^{(6,4,1,1)} \oplus \cs^{(6,3,3)}\oplus \cs^{(6,3,2,1)}\oplus \cs^{(5,5,2)} \oplus
\cs^{(5,4,3)}\oplus \cs^{(5,4,2,1)}\\& & \oplus \cs^{(5,3,3,1)}\oplus \cs^{(4,3,3,2)} \\
 \hline 13 & 2 & \cs^{(7,5,1)} \oplus \cs^{(7,3,3)}\oplus \cs^{(6,5,2)}\oplus \cs^{(5,5,3)} \\
 \hline 13 & 3 & \cs^{(9,1,1,1,1)}\oplus \cs^{(7,3,1,1,1)}\oplus \cs^{(5,5,1,1,1)}
 \oplus \cs^{(5,3,3,1,1)} \oplus \cs^{(3,3,3,3,1)} \\
 \hline
\end{array}
\]

Upon looking at this table one might be tempted to conjecture that whenever $n \equiv 1 \mod p$, the Specht modules in the decomposition of each homology (not just the top homology)  are
multiplicity free.  Indeed, this holds when $p=2$ by Bouc's result (\ref{bouc}).  However, this is not the case, as the multiplicity of $\cs^{(7,5,3,1)}$ in $\rh_3(\mathsf M_{3}(16))$ is two.  Also, while it is the case that for each $k \leq 4$, every Specht module appearing as a submodule of $\rh_{k-2}(\mathsf M_{3}(3k+1))$ is indexed by a partition with $k-1$ parts, this fails for $k=5$, as $\rh_3(\mathsf M_{3}(16))$ has a submodule isomorphic to $\cs^{(7,6,3)}$.  (Similarly, using (\ref{bouc}), one sees that for $k \leq 5$, every Specht module appearing as a submodule of $\rh_{k-2}(\mathsf M_{2}(2k+1))$ is indexed by a partition with $k-1$ parts, but this fails for $k=6$.)  The given examples involving $\rh_3(\mathsf M_{3}(16))$ can be derived using the technique described above.

\end{document}